\documentclass[preprint]{elsarticle}

 \usepackage{amssymb}
 \usepackage{amsmath}
 \usepackage{amsthm}
 \usepackage{enumerate}
 \usepackage{graphicx}
 \usepackage{amsfonts}
 \usepackage{minipage}
 \usepackage{float}
 \newtheorem{example}{Example}

 \journal{...}

 \newtheorem{theorem}{Theorem}[section]
 
 \newtheorem{lemma}{Lemma}[section]
 
 \newtheorem{cor}{Corollary}[section]
 
 \newtheorem{remark}{Remark}[section]

 \theoremstyle{definition}

\begin{document}

\begin{frontmatter}

\title{Better numerical approximation by  Durrmeyer type operators}

\author[1]{Ana-Maria Acu}
\author[2]{Vijay Gupta}
\author[3]{Gancho Tachev}
\author[]{\hspace{10.5cm} { \bf Dedicated to professor Heiner Gonska on the occasion of his 70th birthday}}
\address[1]{Lucian Blaga University of Sibiu, Department of Mathematics and Informatics, Str. Dr. I. Ratiu, No.5-7, RO-550012  Sibiu, Romania, e-mail: anamaria.acu@ulbsibiu.ro}
\address[2]{ Netaji Subhas Institute of Technology, Department of Mathematics, Sector 3 Dwarka, New Delhi-110078, India, e-mail: vijaygupta2001@hotmail.com}
\address[3]{University of Architecture Civil Engineering and Geodesy, Department of Mathematics,
	BG-1046, Sofia, Bulgaria, e-mail: gtt\_fte@uacg.bg}

\begin{abstract}
The main object of this paper is to construct new Durrmeyer type operators which have better features than the classical one. Some results concerning the rate of convergence and asymptotic formulas of the new operator are given. Finally, the theoretical results are analyzed by numerical examples.
\end{abstract}

\begin{keyword}
	Approximation by polynomials, Durrmeyer polynomials, Voronovskaja type theorem
	\MSC[2010] 41A25, 41A36.
\end{keyword}

\end{frontmatter}

\section{Introduction}

In 1912, Bernstein \cite{Bernstein} defined the Bernstein polynomials in order to prove Weierstrass's fundamental theorem. Because of their remarkable and notable approximation properties, Bernstein polynomials attracted the most interest and were studied  by a number of authors. For more details on this topic we can refer the readers to excellent monographs \cite{G1} and \cite{G2}. The order of approximation of Bernstein operators has been studied in great detail for a long time. Firstly, T. Popoviciu (\cite{P43,P44}) gave a solution in this direction considering the first modulus of continuity. Later these results have been improved by Lorentz \cite{32r3} and Sikkema \cite{58r3} and the recent monograph of  Bustamante \cite{bust}. New technique introduced in order to study the order of approximation was given by Esser \cite{11r3}, who used the first estimates with $\omega_2$. The order of approximation of Bernstein operators by $\omega_2$ has been established by Gonska \cite{21r2}, Gonska $\&$ Zhou \cite{27r2}, Kacs\'o \cite{30r2} and P\u alt\u anea \cite{50r2}. An asymptotic error term of the Bernstein operators was first given by Voronovskaja \cite{T9}. Later, this result was extended by the authors of \cite{Gavrea,2,3, G1, G2,5}.

 For $f\in C[0,1]$, the Bernstein operators of degree $n$ with respect to $f$ are defined by
\begin{eqnarray*}
B_n(f,x) = \sum\limits_{k=0}^{n} p_{n,k}(x) \ f\left(\frac{k}{n} \right) , \ \ x \in [0,1],
\end{eqnarray*}
where $p_{n,k}(x)=\displaystyle \binom{n}{k} x^k (1-x)^{n-k}, \ \ k=0,1, \ldots, n, $
and  $p_{n,k}(x)=0,$ if $k<0$ or $k>n.$
It is well known that the fundamental polynomial verify
\begin{eqnarray}\label{e2}
p_{n,k}(x) = (1-x) \ p_{n-1,k}(x) + x \ p_{n-1,k-1}(x), \ \ 0<k< n.
\end{eqnarray}

The most attention arises of how to introduce new modification of Bernstein operators presenting better convergence. In a recent paper, H. Khosravian-Arab et al. \cite{1} have introduced a modified Bernstein operators to improve the degree of approximation as follows:
\begin{eqnarray}\label{ne1}
B_n^{M,1}(f,x) = \sum\limits_{k=0}^{n} p_{n,k}^{M,1}(x) \ f\left(\frac{k}{n} \right) , \ \ x \in [0,1],
\end{eqnarray}
\begin{align}\label{e3}
p_{n,k}^{M,1}(x) &= a(x,n) \ p_{n-1,k}(x) + a(1-x,n) \ p_{n-1,k-1}(x), 1\leq k\leq n-1,\\
p_{n,0}^{M,1}(x) &= a(x,n)(1-x)^{n-1},\quad p_{n,n}^{M,1}(x) = a(1-x,n)x^{n-1},\notag
\end{align}
and
\begin{eqnarray*}\label{e4}
a(x,n) = a_1(n) \ x + a_0(n), \,  n=0,1, \ldots,
\end{eqnarray*}
where $a_0(n)$ and $a_1(n)$ are two unknown sequences which are determined in appropriate way. For $a_1(n)=-1,$ $a_0(n)=1$, obviously, (\ref{e3}) reduces to (\ref{e2}).

The classical Durrmeyer operators are the integral modification of Bernstein operators so as to approximate Lebesque integrable functions defined on the interval $[0,1]$. These operators were introduced by Durrmeyer \cite{1a} and, independently, by Lupa\c s \cite{2a}. From the approximation point of view Durrmeyer operators attracted attention and have been extensively studied by Derrienic \cite{8r3}, Ditzian \cite{10r3}, Ivanov \cite{22r3} and several other authors. In this paper the Durrmeyer variants of the Bernstein operators modified by H. Khosravian-Arab et al.  \cite{1} will be introduced and studied.

\section{Approximation by Durrmeyer operators of order I}
In this section we define a Durrmeyer variant of  the modified
Bernstein operators (\ref{ne1}) as follows:

\begin{eqnarray}\label{X}
D_n^{M,1}(f,x) =(n+1) \sum\limits_{k=0}^{n} p_{n,k}^{M,1}(x)  \int\limits_{0}^{1} p_{n,k}(t) \ f(t) \, dt , \ \ x \in [0,1].
\end{eqnarray}

\begin{lemma}\label{l1.1}
The moments of the Durrmeyer operators $D_n^{M,1}$ are given by	
	\begin{itemize}
		\item[i)]$D_n^{M,1}(e_0;x)=2a_0(n)+a_1(n);$
		\item[ii)]$D_n^{M,1}(e_1;x)=\left(a_1(n)+2a_0(n)\right)x+\displaystyle\frac{(1-2x)(3a_0(n)+2a_1(n))}{n+2};$
		\item[iii)]$D_n^{M,1}(e_2;x)=\left(a_1(n)+2a_0(n)\right)x^2-\left[(8x^2-5x)a_0(n)+(5x^2-3x)a_1(n)\right]\displaystyle\frac{2n}{(n+2)(n+3)}\\   \, \, \, \, -2\displaystyle\frac{(x^2+5x-3)a_1(n)+(4x^2+5x-4)a_0(n)}{(n+2)(n+3)},$
	\end{itemize}
where $e_i(x)=x^i,\,i=0,1,2$.
\end{lemma}
\begin{lemma}\label{lema2}
	The central moments of the operators $D_n^{M,1}$ are given by
	\begin{itemize}
		\item[i)]$D_n^{M,1}(t-x;x)=\displaystyle\frac{(1-2x)(2a_1(n)+3a_0(n))}{n+2};$
		\item[ii)] $D_n^{M,1}((t-x)^2;x)=\displaystyle\frac{2(3a_1(n)+4a_0(n)-11a_1(n)x+14x^2a_0(n)+11a_1(n)x^2-14a_0(n)x)}{(n+2)(n+3)}\\+\displaystyle\frac{2(1-x)x(2a_0(n)+a_1(n))n}{(n+2)(n+3)};$
		\item[iii)]$D_n^{M,1}((t-x)^4;x)=\displaystyle\frac{12}{(n+2)(n+3)(n+4)(n+5)}\left\{(1-x)^2x^2(2a_0(n)+a_1(n))n^2\right.\\
		+3(1-x)x(15a_1(n)x^2+22x^2a_0(n)-15a_1(n)x-22a_0(n)x+6a_0(n)+4a_1(n))n\\
		+10a_1(n)+12a_0(n)-68a_1(n)x+124x^4a_0(n)+202x^2a_0(n)+114x^4a_1(n)+182a_1(n)x^2\\
		\left. -78a_0(n)x-228a_1(n)x^3-248a_0(n)x^3\right\}.$
	\end{itemize}
\end{lemma}

Note that throughout the paper we will assume that the
 sequences $a_i(n), i=0,1$ verify the condition
\begin{equation}\label{A}
2a_0(n)+a_1(n)=1.
\end{equation}
This assumption on the sequences  $a_i(n), i=0,1$  was made in order to study the uniformly convergence.
In the following we will consider these two cases for unknown sequences $a_0(n)$ and $a_1(n)$:

\noindent {\bf Case 1.} Let
\begin{equation}\label{Y1}
a_0(n)\geq 0, a_0(n)+a_1(n)\geq 0.
\end{equation}
Using condition (\ref{A})  we obtain $0\leq a_0(n)\leq 1$ and $-1\leq a_1(n)\leq 1$, namely the sequences are bounded. In this case the operator (\ref{X})  is positive.

\noindent{\bf Case 2.} Let
\begin{equation}\label{Y2}
a_0(n)<0 \textrm{ or } a_1(n)+a_0(n)<0.
\end{equation}
If $a_0(n)<0$, then $a_1(n)+a_0(n)>1$ and if $a_1(n)+a_0(n)<0$, then $a_0(n)>1$. In this case the operator (\ref{X})  is not positive.

\begin{theorem}\label{t1.1}
 Let $f\in C[0,1]$. If $a_1(n)$, $a_0(n)$  verify the conditions (\ref{A}) and (\ref{Y1}), then
\begin{align*}
&\displaystyle\lim_{n\to\infty}D_n^{M,1}(f;x)=f(x),
\end{align*}
uniformly on $[0,1]$.
\end{theorem}
\begin{proof} Since the sequences $a_1(n)$, $a_0(n)$  verify the conditions (\ref{A}) and (\ref{Y1}), it follows that these sequences are bounded. Using the well known Korovkin Theorem and Lemma \ref{l1.1},  the uniform convergence of the operators $D_n^{M,1}$ is proved.\end{proof}

 In order to prove this result for the Case 2, we recall the extended form of the Korovkin Theorem:
\begin{theorem}\label{t10}
\cite[Theorem 10]{1} Let $0<h\in C[a,b]$ be a function and suppose that $(L_n)_{n\geq 1}$ is a sequence of positive linear operators such that
$\displaystyle\lim_{n\to\infty} L_n(e_i)=he_i, i=0,1,2$, uniformly on [a,b]. Then for given function $f\in C[a,b]$ we have $\displaystyle\lim_{n\to\infty} L_n(f)=hf$ uniformly on $[a,b]$.
\end{theorem}
\begin{theorem}\label{t1.3}
Let $f\in C[0,1]$. Then for all bounded sequences $a_1(n)$ and $a_0(n)$ that satisfy the conditions (\ref{A}) and (\ref{Y2}), we have
\begin{align*}
&\displaystyle\lim_{n\to\infty}D_n^{M,1}(f;x)=f(x),
\end{align*}
uniformly on $[0,1]$.
\end{theorem}
\begin{proof} The operators $D_n^{M,1}$ can be written as follows:
	$$ D_n^{M,1}(f;x)=D_{n,2}^{M,1}(f;x)-D_{n,1}^{M,1}(f;x),  $$
	where
\begin{align*}
& D_{n,1}^{M,1}(f;x)=(n+1)\sum_{k=0}^n\left[-a_1(n)xp_{n-1,k}(x)-a_1(n)p_{n-1,k-1}(x)\right]\int_{0}^{1}p_{n,k}(t)f(t)dt;\\
& D_{n,2}^{M,1}(f;x)=(n+1)\sum_{k=0}^n\left[a_0(n)p_{n-1,k}(x)+\left(-a_1(n)x+a_0(n)\right)p_{n-1,k-1}(x)\right]\int_{0}^{1}p_{n,k}(t)f(t)dt.
\end{align*}
Since the sequences $a_i(n),i=0,1$ verify the conditions (\ref{A}) and (\ref{Y2}), it follows $(a_0(n)<0,\,\,a_1(n)>0)$ or $(a_0(n)>0,a_1(n)<0)$.  Therefore, the extended form of the Korovkin theorem can be applied to  $D_{n,1}^{M,1}$ and $D_{n,2}^{M,1}$.
Below are calculated the  moments of the operators $D_{n,i}^{M,1}$:
\begin{align*}
& D_{n,1}^{M,1}(e_0;x)=-a_1(n)(x+1);\,
D_{n,1}^{M,1}(e_1;x)=\displaystyle -\frac{a_1(n)\left(nx^2+nx-x^2+2\right)}{n+2};\\
& D_{n,1}^{M,1}(e_2;x)=\displaystyle-a_1(n)\frac{x^2(x+1)n^2-x(3x^2-x-6)n+2x^3-2x^2-4x+6}{(n+2)(n+3)};\\ & D_{n,2}^{M,1}(e_0;x)=2a_0(n)-a_1(n)x;\\ &
 D_{n,2}^{M,1}(e_1;x)=\displaystyle\frac{\left[2a_0(n)-a_1(n)x\right]nx+a_1(n)x^2-2a_0(n)x-2a_1(n)x+3a_0(n)}{n+2};\\
& D_{n,2}^{M,1}(e_2;x)=\displaystyle \frac{(2a_0(n)-a_1(n)x)x^2n^2+(3a_1(n)x^3-6a_0(n)x^2-6a_1(n)x^2+10a_0(n)x)n}{(n+2)(n+3)}\\
&+\displaystyle\frac{-2a_1(n)x^3+4a_0(n)x^2+6a_1(n)x^2-10a_0(n)x-6a_1(n)x+8a_0(n)}{(n+2)(n+3)} .
\end{align*}
Using the above relations and Theorem \ref{t10} we obtain
\begin{align*}
&\displaystyle\lim_{n\to\infty} D_{n,1}^{M,1}(f;x)=-l_1(1+x)f(x),\\
&\displaystyle\lim_{n\to\infty}D_{n,2}^{M,1}(f;x)= (1-l_1(1+x))f(x),
\end{align*}
where $l_1=\displaystyle\lim_{n\to\infty}a_1(n)$.
Therefore  $\displaystyle\lim_{n\to\infty}D_n^{M,1}(f;x)=f(x)$.
\end{proof}
\begin{theorem}\label{t2.4}
Let $a_i(n)$ be a convergent sequence that satisfies the conditions (\ref{A}) and (\ref{Y1}) and $l_i=\displaystyle\lim_{n\to\infty}a_i(n)$, $i=0,1$. If $f^{\prime\prime}\in C[0,1]$, then
$$
\displaystyle\lim_{n\to\infty} n\left(D_n^{M,1}(f;x)-f(x)\right)=(1-2x)(2l_1+3l_0)f^\prime(x)+x(1-x) (2l_0+l_1)f^{\prime\prime}(x),$$
uniformly on $[0,1]$.
\end{theorem}
\begin{proof}
Applying the Durrmeyer operators $D_n^{M,1}$ to the Taylor's formula, we obtain
$$D_n^{M,1}(f;x)-f(x)=D_n^{M,1}(t-x;x)f^{\prime}(x)+\displaystyle\frac{1}{2}D_n^{M,1}\left((t-x)^2;x\right)f^{\prime\prime}(x)+D_n^{M,1}\left(\theta(t,x)(t-x)^2;x\right),  $$	
where $\theta\in C[0,1]$ and $\displaystyle\lim_{t\to x}\theta(t,x)=0$.

Using the Cauchy-Schwarz inequality for the positive operators $D_n^{M,1}$, we get
$$ D_n^{M,1}\left(\theta(t,x)(t-x)^2;x\right)\leq\sqrt{D_n^{M,1}\left(\theta^2(t,x);x\right)} \sqrt{D_n^{M,1}\left((t-x)^4;x\right)}. $$
Since $\theta^2(x,x)=0$ and $\theta^2(\cdot,x)\in C[0,1]$, by Theorem \ref{t1.1}, we obtain
$$ \displaystyle\lim_{n\to\infty}D_n^{M,1}\left(\theta^2(t,x);x\right)=0 $$
uniformly with respect to $x\in[0,1]$. Therefore, from Lemma \ref{lema2} we obtain
$$ \displaystyle\lim_{n\to\infty}nD_n^{M,1}\left(\theta(t,x)(t-x)^2;x\right)=0. $$
Using the results from Lemma \ref{lema2}, the proof of this theorem is completed.
\end{proof}

Now we extend the results from Theorem \ref{t2.4} when the operator $D_n^{M,1}$ is nonpositive, i.e. the sequences $a_0(n)$ and $a_1(n)$ satisfy \eqref{A} and \eqref{Y2}.

\begin{theorem}\label{t2.5} Let $a_i(n), i=0,1$ be bounded convergent sequences which satisfy \eqref{A} and \eqref{Y2} and $l_i=\displaystyle\lim_{n\to\infty} a_i(n), i=0,1.$ If $f\in C[0,1]$ and $f^{\prime\prime}$ exists at a certain point $x\in[0,1]$, then we have
\begin{align}
\displaystyle\lim_{n\to\infty} n[D_{n}^{M,1}(f;x)-f(x)]&=(1-2x)(2l_1+3l_0)f^\prime(x)+x(1-x) (2l_0+l_1)f^{\prime\prime}(x).\label{eG1}
\end{align}
Moreover the relation \eqref{eG1} holds uniformly on $[0,1]$ if $f^{\prime\prime}\in C[0,1].$
\end{theorem}
\begin{proof}
	In the same manner as the proof of previous theorem, applying the Durrmeyer operator $ D_n^{M,1}$ to the Taylor's formula it is sufficient to show that
\begin{align}
\displaystyle\lim_{n\to\infty} nD_{n}^{M,1}(\theta(t,x)(t-x)^2;x)&=0.\label{eG2}
\end{align}
Since the operators $ D_n^{M,1}$ are not positive linear operators we will introduce new techniques in order to prove the theorem.
From \eqref{e3} and \eqref{X} we may represent $D_{n}^{M,1}$ as
\begin{align}
D_{n}^{M,1}(f;x)&=(n+1)\sum_{k=0}^{n-1}\left(a(x,n)\int_0^1p_{n,k}(t)f(t)dt+a(1-x,n)\int_0^1p_{n,k+1}(t)f(t)dt\right)p_{n-1,k}(x).\label{eG3}
\end{align}
Let $\varepsilon>0$ be given. There exist a $\delta>0$ such that if $|t-x|<\delta$ then $|\theta(t,x)|<\varepsilon$. We denote
$${\mathcal K}_1=\left\{k:\left|\frac{k}{n}-x\right|<\delta, k=0,1,2,\cdots ,n\right\},\quad {\mathcal K}_2=\left\{k:\left|\frac{k}{n}-x\right|\ge \delta, k=0,1,2,\cdots ,n\right\}.$$
The boundedness of the sequences $a_i(n)$, $i=0,1$ implies that these is a constant $C>0$ such that $|a(x,n)|<C$. Then, it follows

\begin{align}
\left|D_{n}^{M,1}(\theta(t,x)(t-x)^2;x)\right|&\leq (n+1)C\displaystyle\sum_{k=0}^{n-1}p_{n-1,k}(x)\left\{\int_0^1 p_{n,k}(t)|\theta(t,x)|(t-x)^2 dt\right.\nonumber\\
&+\left.\int_0^1 p_{n,k+1}(t)|\theta(t,x)|(t-x)^2 dt\right\}. \label{eG4}
\end{align}
Let $k\in {\mathcal K}_1$. Hence $|\theta(t,x)|<\varepsilon$. Therefore we get
\begin{align}
\left|D_{n}^{M,1}(\theta(t,x)(t-x)^2;x)\right|&\leq (n\!+\!1)C\varepsilon\displaystyle\sum_{k=0}^{n-1}p_{n-1,k}(x)\left\{\int_0^1 p_{n,k}(t)(t\!-\!x)^2 dt\!+\!\int_0^1 p_{n,k+1}(t)(t\!-\!x)^2 dt\right\}\nonumber\\
&=\displaystyle\frac{4\varepsilon C}{(n+2)(n+3)}\left(nx(1-x)+7x^2-7x+2\right).\label{ec15}
\end{align}

Let $k\in {\mathcal K}_2$. We denote $M=\displaystyle\sup_{0\leq t\leq 1}|\theta(t,x)|(t-x)^2$. Then $|\theta(t,x)|(t-x)^2\leq\displaystyle \frac{M}{\delta^4}\left(\frac{k}{n}-x\right)^4$. From (\ref{eG4}) we get the following upper bound
\begin{align}\label{ec16}
&\left|D_{n}^{M,1}(\theta(t,x)(t-x)^2;x)\right|\leq \displaystyle\frac{2MC}{\delta^4}\sum_{k=0}^{n-1} p_{n-1,k}(x)\left(\frac{k}{n}-x\right)^4\nonumber\\
&=\displaystyle\frac{2MC}{\delta^4n^4}x\left[3x(1-x)^2n^2+(1-x)(26x^2-16x+1)n+(2x-1)(12x^2-12x+1)\right].
\end{align}
Using (\ref{ec15}) and (\ref{ec16}), it follows
\begin{align*}
&\left|D_{n}^{M,1}(\theta(t,x)(t-x)^2;x)\right|\leq\displaystyle\frac{4\varepsilon C}{(n+2)(n+3)}\left(nx(1-x)+7x^2-7x+2\right)\nonumber\\
&+\displaystyle\frac{2MC}{\delta^4n^4}x\left[3x(1-x)^2n^2+(1-x)(26x^2-16x+1)n+(2x-1)(12x^2-12x+1)\right].
\end{align*}
Therefore, the last inequality leads to (\ref{eG2}) and the proof  is completed.
\end{proof}
The next result is a direct estimate for the Durrmeyer operators $D_n^{M,1}$:
\begin{theorem}\label{t1.6} If $f$ is a bounded function for $x\in [0,1]$,   $a_1(n)$ is a bounded sequence and $a_0(n)$, $a_1(n)$ satisfy (\ref{A}), then
\begin{equation*}
\| D_n^{M,1}f-f\|\leq \displaystyle\left(3|a_1(n)|+1 \right)\left(1+\sqrt{2}\right)\omega\left(f;\frac{1}{\sqrt{n}}\right),\,\, n\geq 1,
\end{equation*}
where $\|\cdot \|$ is the uniform norm on the interval $[0,1]$ and $\omega(f,\delta)$ is the first order modulus of continuity.
\end{theorem}
\begin{proof}Using Lemma \ref{l1.1}, condition (\ref{A}) and representation (\ref{eG3}), we get
\begin{align*}
\left|D_{n}^{M,1}(f;x)-f(x)\right|&\leq (n+1)|a(x,n)|\displaystyle\sum_{k=0}^{n-1}p_{n-1,k}(x)\int_0^1p_{n,k}(t)|f(t)-f(x)|dt\\
&+(n+1)|a(1-x,n)|\displaystyle\sum_{k=0}^{n-1}p_{n-1,k}(x)\int_0^1p_{n,k+1}(t)|f(t)-f(x)|dt
\nonumber\\
&\leq (n+1)|a(x,n)|\displaystyle\sum_{k=0}^{n-1}p_{n-1,k}(x)\int_0^1 p_{n,k}(t)\omega(f;|t-x|)dt\\
&+ (n+1)|a(1-x,n)|\displaystyle\sum_{k=0}^{n-1}p_{n-1,k}(x)\int_0^1 p_{n,k+1}(t)\omega(f;|t-x|)dt.
\end{align*}
Using the following well known relation for the first order modulus of continuity
\begin{equation*}
\omega(f;|t-x|)\leq \left( 1+\sqrt{n}|t-x|\right)\omega\left(f;\frac{1}{\sqrt{n}}\right),
\end{equation*}
we get
\begin{align*}
\left|D_{n}^{M,1}(f;x)-f(x)\right|&\leq |a(x,n)|\omega\left(f;\frac{1}{\sqrt{n}}\right)\left\{1+(n+1)\sqrt{n}\sum_{k=0}^{n-1}p_{n-1,k}(x)\int_0^1p_{n,k}(t)|t-x|dt\right\}\\
&+|a(1-x,n)|\omega\left(f;\frac{1}{\sqrt{n}}\right)\left\{1+(n+1)\sqrt{n}\sum_{k=0}^{n-1}p_{n-1,k}(x)\int_0^1p_{n,k+1}(t)|t-x|dt\right\}.
\end{align*}
From H\"{o}lder's inequality it follows:
\begin{align}\label{ec21}
&\sum_{k=0}^{n-1}p_{n-1,k}(x)\int_0^1p_{n,k}(t)|t-x|dt\leq \displaystyle\sum_{k=0}^{n-1}p_{n-1,k}(x)\left[\int_0^1p_{n,k}(t)dt\right]^{\frac{1}{2}}\left[\int_0^1p_{n,k}(t)(t-x)^2\right]^{\frac{1}{2}}\nonumber\\
&=\displaystyle\frac{1}{\sqrt{n+1}}\sum_{k=0}^{n-1}p_{n-1,k}(x)\left[\int_0^1p_{n,k}(t)(t-x)^2dt\right]^{\frac{1}{2}}\nonumber\\
&=\displaystyle\frac{1}{\sqrt{n+1}}\left[\sum_{k=0}^{n-1}p_{n-1,k}(x)\right]^{\frac{1}{2}}\left[\sum_{k=0}^{n-1}p_{n-1,k}(x)\int_0^1p_{n,k}(t)(t-x)^2dt\right]^{\frac{1}{2}}\nonumber\\
&=\displaystyle\frac{1}{\sqrt{n+1}}\sqrt{\frac{-2nx^2+2nx+14x^2-10x+2}{(n+1)(n+2)(n+3)}}\leq\frac{\sqrt{2}}{(n+1)\sqrt{n+2}}.
\end{align}
Again as in the prove of the previous relation, we get
\begin{equation}\label{p}
\sum_{k=0}^{n-1}p_{n-1,k}(x)\int_0^1p_{n,k+1}(t)|t-x|dt\leq \frac{\sqrt{2}}{(n+1)\sqrt{n+2}} .
\end{equation}
Consequently, the inequalities (\ref{ec21}) and  (\ref{p}) lead to the following relation:
\begin{align*}
\left|D_{n}^{M,1}(f;x)-f(x)\right|&\leq\left( \left| a(x,n)\right|+ \left| a(1-x,n)\right|\right) \omega\left(f;\frac{1}{\sqrt{n}}\right)\left(1+\sqrt{2}\right).%\label{ec22}
\end{align*}
It is clear from (\ref{A}) that
\begin{equation}\label{ec23}
|a(x,n)|=|a_1(n)x+a_0(n)|\leq |a_1(n)|+|a_0(n)|=|a_1(n)|+\left|\frac{1-a_1(n)}{2}\right|\leq \frac{3}{2}|a_1(n)|+\frac{1}{2}.
\end{equation}
The same upper bound (\ref{ec23}) holds for $|a(1-x,n)|$. Therefore,
\begin{equation*}
\left|D_{n}^{M,1}(f;x)-f(x)\right|\leq\displaystyle\left(3|a_1(n)|+1\right)\left(1+\sqrt{2}\right)\omega\left(f;\frac{1}{\sqrt{n}}\right).
\end{equation*}
The proof of Theorem \ref{t1.6} is completed.
\end{proof}

\begin{remark}
We point out that in \cite[Theorem 9]{1} it was supposed that the modified Bernstein operator $B_n^{M,1}$ is positive, although the proof holds true also for the cases when the operator is nonpositive. Of course we need that the sequence $a_1(n)$ is bounded in the last case.
\end{remark}

\begin{cor} If $f\in C[0,1]$, then $\displaystyle\lim_{n\to\infty}\omega\left(f,\frac{1}{\sqrt{n}}\right)=0$, and thus another proof of the Theorems \ref{t1.1} and  \ref{t1.3} is given.
\end{cor}

A function $f$ satisfies the Lipschitz condition of order $r$ with the constant $k$ on $[0,1]$, i.e.
$$ |f(x_1)-f(x_2)|\leq k|x_1-x_2|^r,\, 0<r\leq 1,\, k>0 \textrm{ for } x_1,x_2\in [0,1] $$
if and only if $\omega(\delta)\leq k\delta^r$.

\begin{cor} If $f$ satisfies the Lipschitz condition of order $r$ with constant $k$, then
\begin{equation*}
\|D_n^{M,1}f-f\|\leq \displaystyle\left(3|a_1(n)|+1\right)\left(1+\sqrt{2}\right)k n^{-\frac{r}{2}}.
\end{equation*}
\end{cor}

Next we will give a quantitative form of Voronovskaja type theorem for the Durrmeyer operator $D_n^{M,1}$ in the case when $e_0,e_1$ are reproduced, i.e
$$
2a_0(n)+a_1(n)=1,\,\,
3a_0(n)+2a_1(n)=0.
$$
Hence $a_0(n)=2$ and $a_1(n)=-3$. From (\ref{Y2}) it follows that $D_n^{M,1}$ is non-positive operator. The advantage in this case  is that now $D_n^{M,1}$ has order of approximation better than $O\left(\frac{1}{n}\right)$ for $f\in C^2[0,1]$.
\begin{theorem}\label{t2.7}
	Let $g\in C^2[0,1]$, $x\in [0,1]$ fixed and $D_n^{M,1}$ defined above. Then,
	\begin{equation*}
	\left| D_n^{M,1}(g;x)-g(x)-\frac{1}{2}D_n^{M,1}\left((t-x)^2;x\right)g^{\prime\prime}(x)\right|\leq C\frac{1}{n}\omega\left(g^{\prime\prime},\frac{1}{\sqrt{n}}\right),
	\end{equation*}
	where $C>0$ is a constant independent of $n$, $x$.
\end{theorem}
\begin{proof}
Let $g,g^{\prime\prime}\in C[0,1]$. Applying $D_n^{M,1}$ to the Taylor's formula, we get
\begin{equation}
\label{e27}
D_{n}^{M,1}(g;x)-g(x)-\displaystyle\frac{1}{2}D_n^{M,1}\left((t-x)^2;x\right)g^{\prime\prime}(x)=\frac{1}{2}D_n^{M,1}\left(\theta(t,x)(t-x)^2;x\right),
\end{equation}
where $\theta(t,x)=g^{\prime\prime}(\xi_x)-g^{\prime\prime}(x)$ and $\xi_x$ is a point between $t$ and $x$. Therefore,
$$ |\theta(t,x)|\leq \omega\left(g^{\prime\prime},|t-x|\right)\leq \left(1+\sqrt{n}|t-x|\right)\omega\left(g^{\prime\prime},\frac{1}{\sqrt{n}}\right). $$
Since $|a_1(n)x+a_0(n)|\leq \displaystyle 2$ for $x\in [0,1]$, using  the relation (\ref{eG3}),  we have
\begin{align}\label{o}
\left|D_n^{M,1}\left(\theta(t,x)(t-x)^2;x\right)\right| &\leq \displaystyle 2(n+1)\omega\left(g^{\prime\prime},\frac{1}{\sqrt{n}}\right)\left(S_1+\sqrt{n}S_2\right),
\end{align}
where
\begin{align*}
S_1&:=\sum_{k=0}^{n-1}p_{n-1,k}(x)\left(\int_0^1p_{n,k}(t)(t-x)^2dt+\int_0^1p_{n,k+1}(t)(t-x)^2 dt\right)\\
&=\displaystyle\frac{4}{(n+1)(n+2)(n+3)}\left[nx(1-x)+7x^2-7x+2\right]\leq \frac{n+8}{(n+1)(n+2)(n+3)},\\
S2&:=\displaystyle\sum_{k=0}^{n-1}p_{n-1,k}(x)\left(\int_0^1p_{n,k}(t)|t-x|(t-x)^2dt+\int_0^1p_{n,k+1}(t)|t-x|(t-x)^2dt\right)\\
&=\displaystyle\sum_{k=0}^{n-1}p_{n-1,k}(x)\left[\int_0^1p_{n,k}(t)(t-x)^2dt\right]^{\frac{1}{2}}\cdot\left[\int_0^1p_{n,k}(t)(t-x)^4dt\right]^{\frac{1}{2}}\\
&+\displaystyle\sum_{k=0}^{n-1}p_{n-1,k}(x)\left[\int_0^1p_{n,k+1}(t)(t-x)^2dt\right]^{\frac{1}{2}}\cdot\left[\int_0^1p_{n,k+1}(t)(t-x)^4dt\right]^{\frac{1}{2}}\\
&=\left[\displaystyle\sum_{k=0}^{n-1}p_{n-1,k}(x)\int_0^1p_{n,k}(t)(t-x)^2dt\right]^{\frac{1}{2}}\cdot\left[\sum_{k=0}^{n-1}p_{n-1,k}(x)\int_0^1p_{n,k}(t)(t-x)^4dt\right]^{\frac{1}{2}}\\
&+\left[\displaystyle\sum_{k=0}^{n-1}p_{n-1,k}(x)\int_0^1p_{n,k+1}(t)(t-x)^2dt\right]^{\frac{1}{2}}\cdot\left[\sum_{k=0}^{n-1}p_{n-1,k}(x)\int_0^1p_{n,k+1}(t)(t-x)^4dt\right]^{\frac{1}{2}}\\
&=\displaystyle\frac{2\sqrt{6}}{(n+1)(n+2)(n+3)\sqrt{(n+4)(n+5)}}\left\{(x(1-x)n+7x^2-5x+1)^{\frac{1}{2}}\right.\\
&\cdot\left[x^2(1-x)^2n^2+3x(1-x)(11x^2-9x+2)n+62x^4-98x^3+62x^2-18x+2\right]^{\frac{1}{2}}\\
&+\left[x(1-x)n+7x^2-9x+3\right]^{\frac{1}{2}}\cdot\left[x^2(1-x)^2n^2+3x(1-x)(11x^2-13x+4)n+62x^4\right.\\
&\left.\left.-150x^3+140x^2-60x+10\right]^{\frac{1}{2}}
\right\}.
\end{align*}
Replacing $S_1$ and $S_2$ in relation (\ref{o}) we conclude that there is a positive constant $C$ independent of $n$, $x$ such that
\begin{equation}\label{Xt}	\left|D_n^{M,1}\left(\theta(t,x)(t-x)^2;x\right)\right|\leq C\frac{1}{n}\omega\left(g^{\prime\prime},\frac{1}{\sqrt{n}}\right).\end{equation}
Consequently, the proof is completed.

\end{proof}

The next goal is to extend the direct estimate in Theorem \ref{t1.6} in terms of second order moduli $\omega_2\left(f;\frac{1}{\sqrt{n}}\right).$

\begin{theorem}\label{t2.8}
If $f\in C[0,1]$, $a_0(n)=2$, $a_1(n)=-3$, then
\begin{equation*}
 \|D_n^{M,1}f-f\|_{C[0,1]}\leq C\omega_2\left(f;\frac{1}{\sqrt{n}}\right).
\end{equation*}
\end{theorem}
\begin{proof}
From (\ref{eG3})	 we observe that $D_n^{M,1}:C[0,1]\to C[0,1]$ is a bounded operator, namely
	\begin{equation*}
		\|D_n^{M,1}f\|_{C[0,1]}\leq \left(|a_0(n)|+|a_1(n)|\right)\|f\|_{C[0,1]}.
	\end{equation*}
	Using relations (\ref{e27}) and (\ref{Xt}) we can write
	\begin{equation}
	\label{Y}
	 \left|D_n^{M,1}(g;x)-g(x)\right|\leq\displaystyle\frac{1}{2}D_n^{M,1}\left((t-x)^2;x\right)|g^{\prime\prime}(x)|+C_1\frac{1}{n}\omega\left(g^{\prime\prime},\frac{1}{\sqrt{n}}\right),
	\end{equation}
where $C_1$ is a positive constant independent of $n$ and $x$.
From (\ref{Y}), Lemma \ref{lema2} and the property $\omega(h,\delta)\leq 2\| h\|_{C[0,1]}$ it follows
$$ \|D_{n}^{M,1}g-g\|_{C[0,1]}\leq C_2\frac{1}{n}\|g^{\prime\prime}\|_{C[0,1]}, $$	
where $C_2$ is a positive constant independent of $n$ and $x$.
Therefore,
\begin{align}
&\|D_n^{M,1}f-f \|_{C[0,1]}\leq\| D_n^{M,1}(f-g)-(f-g) \|_{C[0,1]}+\| D_n^{M,1}g-g \|_{C[0,1]}\nonumber\\
&\leq C_3\|f-g\|_{C[0,1]}+C_2\frac{1}{n}\| g^{\prime\prime}\|_{C[0,1]}\leq C\left\{\|f-g  \|+\frac{1}{n}\| g^{\prime\prime}\|  \right\},\label{V}
\end{align}	
where $C$ and $C_3$ are some positive constants independent of $n$ and $x$.
	
It is known that the second order moduli $\omega_2(f,t)$ is  equivalent to the following $K$-functional
$$ K_2(f,t^2):=\inf_{f\in C^2[0,1]}\left\{\|f-g\|_{C[0,1]}+t^2\|g^{\prime\prime}\|_{C[0,1]}\right\}. $$
More precisely from \cite[Corrollary 2.7]{G}, we have
\begin{equation*}
K_2(f,t^2)\leq \displaystyle\frac{7}{2}\omega_2(f,t), \, t\geq 0,\, f\in C[0,1].
\end{equation*}

If we take the infimum over all $g\in C^2[0,1]$ in relation (\ref{V}) we complete the proof of Theorem \ref{t2.8}.
\end{proof}
\begin{remark}
Obviously Theorem \ref{t2.8} is better than Theorem \ref{t1.6} because now we have estimate in $\omega_2$ instead of $\omega_1$. Also, we observe that for $e_0, e_1, e_2$ in place of $g$ in Theorem \ref{t2.7} in both sides of the inequality we have 0.
\end{remark}

\section{Approximation by Durrmeyer operators of order II }
In this section we will extend the previous results considering a new Durrmeyer operators that have order of approximation $\displaystyle{\mathcal O}\left(\frac{1}{n^2}\right)$, as follows:
\begin{equation}\label{K2}
D_n^{M,2}(f;x)=(n+1)\displaystyle\sum_{k=0}^np_{n,k}^{M,2}(x)\displaystyle\int_0^1 p_{n,k}(t)f(t)dt,
\end{equation}
where
\begin{equation}
\label{e38p}
p_{n,k}^{M,2}(x)=b(x,n)p_{n-2,k}(x)+d(x,n)p_{n-2,k-1}(x)+b(1-x,n)p_{n-2,k-2}(x) \end{equation}
and
$$ b(x,n)=b_2(n)x^2+b_1(n)x+b_0(n),\,\,d(x,n)=d_0(n)x(1-x),  $$
where $b_i(n),i=0,1,2$ and $d_0(n)$ are two unknown sequences which are determined in appropriate way. For $b_2(n)=b_0(n)=1,$ $b_1(n)=-2$, $d_0(n)=2$ obviously, (\ref{e38p}) reduces to (\ref{e2}).
\begin{lemma}
By simple computation, we have
\begin{itemize}
	\item [i)] $ D_n^{M,2}(e_0;x)=(2b_2(n)-d_0(n))x^2-(2b_2(n)-d_0(n))x+2b_0(n)+b_1(n)+b_2(n);$
	\item [ii)] $ D_n^{M,2}(e_1;x)=\displaystyle\frac{1}{n\!+\!2}\left\{\left[(2b_2(n)\!-\!d_0(n))x^3\!+\!(-2b_2(n)\!+\!d_0(n))x^2\!+\!(2b_0(n)\!+\!b_1(n)\!+\!b_2(n))x\right]n\right.\\
 (-4b_2(n)+2d_0(n))x^3+(8b_2(n)-4d_0(n))x^2+(-4b_0(n)-4b_1(n)-8b_2(n)+2d_0(n))x\\
	\left. +4b_0(n)+3b_1(n)+3b_2(n)	\right\};$
	\item[iii)] $ D_n^{M,2}(e_2;x)=\displaystyle\frac{1}{(n+2)(n+3)}\left\{\left[
	(2b_2(n)-d_0(n))x^4+(-2b_2(n)+d_0(n))x^3
	\right.\right.\\
	\left. +(2b_0(n)+b_1(n)+b_2(n))x^2\right]n^2+\left[(-10b_2(n)+5d_0(n))x^4+(22b_2(n)-11d_0(n))x^3\right.\\
	\left. +(-10b_0(n)-9b_1(n)-21b_2(n)+6d_0(n))x^2+(12b_0(n)+8b_1(n)+8b_2(n))x\right]n\\
	+(12b_2(n)-6d_0(n))x^4+(-36b_2(n)+18d_0(n))x^3+(12b_0(n)+14b_1(n)+52b_2(n)-18d_0(n))x^2\\
	\left.+(-24b_0(n)-26b_1(n)-40b_2(n)+6d_0(n))x+14b_0(n)+12b_1(n)+12b_2(n)\right\}.$
\end{itemize}
\end{lemma}

In order to study the uniform convergence we set $D_n^{M,2}(e_0;x)=1$, and this yields:
$$
2b_2(n)-d_0(n)=0,\,\,
b_2(n)+2b_0(n)+b_1(n)=1.
$$
Using the above relations we obtain
\begin{align*}
D_n^{M,2}(e_1;x)&=x+\displaystyle\frac{(4b_0(n)-6)x+3-2b_0(n)}{n+2};\\
D_n^{M,2}(e_2;x)&=x^2+\displaystyle\frac{2}{(n+2)(n+3)}\left\{4b_0(n)nx^2-2b_0(n)nx-10b_0(n)x^2-b_1(n)x^2-7nx^2\right.\\
& \left. +16b_0(n)x+b_1(n)x+4nx+5x^2-5b_0(n)-14x+6  \right\}.
\end{align*}
In order to have $\displaystyle\lim_{n\to\infty}D_n^{M,2}(e_i;x)=x^i$, $i=0,1,2$ we consider the sequences $b_0(n)$ and $b_1(n)$ to verify the conditions
$$ \displaystyle\lim_{n\to\infty}\frac{b_0(n)}{n}=0\textrm{ and } \displaystyle\lim_{n\to\infty}\frac{b_1(n)}{n^2}=0.  $$
We propose our analysis for the case $b_0(n)=\displaystyle\frac{3}{2}$ and $b_1(n)=-n$, therefore $b_2(n)=n-2$ and $d_0(n)=2(n-2)$.

With the above choices the operator (\ref{K2}) becomes
\begin{equation}\label{TK2} \tilde{D}_n^{M,2}(f;x)=(n+1)\displaystyle\sum_{k=0}^n\tilde{p}_{n,k}^{M,2}(x)\displaystyle\int_0^1p_{n,k}(t)f(t)dt,\end{equation}
where
\begin{align*} \tilde{p}_{n,k}^{M,2}(x)&=\left((n-2)x^2-nx+\frac{3}{2}\right)p_{n-2,k}(x)+2(n-2)x(1-x)p_{n-2,k-1}(x)\\
&+\left((n-2)x^2-(n-4)x-\displaystyle\frac{1}{2}\right)p_{n-2,k-2}(x). \end{align*}
Note that other choices for sequences $b_0(n)$ and $b_1(n)$ lead to some operators with order of approximation either one or two. In the following we are concerning to study the uniform convergence of the operator (\ref{TK2}).

\begin{lemma}
The moments of the  operators (\ref{TK2}) are given by
\begin{itemize}
	\item[i)] $\tilde{D}_n^{M,2}(e_0;x)=1$;
	\item[ii)] $\tilde{D}_n^{M,2}(e_1;x)=x$;
	\item[iii)] $\tilde{D}_n^{M,2}(e_2;x)=x^2+\displaystyle\frac{20x(1-x)}{(n+2)(n+3)}-\displaystyle\frac{3}{(n+2)(n+3)}$.
\end{itemize}
\end{lemma}
\begin{lemma}
The central moments of the operators (\ref{TK2}) are given by
\begin{itemize}
	\item [i)] $\tilde{D}_n^{M,2}\left((t-x)^2;x\right)=\displaystyle\frac{20x(1-x)}{(n+2)(n+3)}-\displaystyle\frac{3}{(n+2)(n+3)}$,
	\item[ii)] $\tilde{D}_n^{M,2}\left((t-x)^3;x\right)=\displaystyle\displaystyle\frac{3(1-2x)}{(n+2)(n+3)(n+4)}\left\{-4nx(1-x)-48x^2+48x-7\right\}$,
	\item[iii)] $\tilde{D}_n^{M,2}\left((t-x)^4;x\right)=-\displaystyle\frac{12x^2(1-x)^2n^2}{(n+2)(n+3)(n+4)(n+5)}+{\cal O}\left(\frac{1}{n^3}\right)$,
	\item[iv)] $\tilde{D}_n^{M,2}\left((t-x)^5;x\right)=-\displaystyle\frac{360(1-2x)x^2(1-x)^2n^2}{(n+2)(n+3)(n+4)(n+5)(n+6)}+{\cal O}\left(\displaystyle\frac{1}{n^4}\right)$,
	\item[v)] $\tilde{D}_n^{M,2}\left((t-x)^6;x\right)=-\displaystyle\frac{240x^3(1-x)^3n^3}{(n+2)(n+3)(n+4)(n+5)(n+6)(n+7)}+{\cal O}\left(\displaystyle\frac{1}{n^4}\right)$.
\end{itemize}
\end{lemma}
\begin{theorem}
If $f\in C^{6}[0,1]$ and $x\in [0,1]$, then for sufficiently large $n$, we have
$$ \tilde{D}_n^{M,2}(f;x)-f(x)={\cal O}\left(\displaystyle\frac{1}{n^2}\right). $$
\end{theorem}
\begin{proof} Applying the Durrmeyer operators $\tilde{D}_n^{M,2}$ to the Taylor's formula, we obtain
$$ \tilde{D}_n^{M,2}(f;x)-f(x)=\displaystyle\sum_{k=1}^6 f^{(k)}(x)\tilde{D}_n^{M,2}\left( (t-x)^k;x\right)+\tilde{D}_n^{M,2}\left(\theta(t,x)(t-x)^6;x\right),  $$
where $\displaystyle\lim_{t\to x}\theta(t,x)=0$.

We have
\begin{align}
\tilde{D}_n^{M,2}(f;x)&=(n+1)\left[(n-2)x^2-nx+\frac{3}{2}\right]\sum_{k=0}^{n-2}p_{n-2,k}(x)\int_0^1p_{n,k}(t)f(t)dt\nonumber\\
&+2(n+1)(n-2)x(1-x)\sum_{k=1}^{n-1}p_{n-2,k-1}(x)\int_0^1p_{n,k}(t)f(t)dt\nonumber\\
&+(n+1)\left[(n-2)x^2-(n-4)x-\frac{1}{2}\right]\sum_{k=2}^np_{n-2,k-2}(x)\int_0^1p_{n,k}(t)f(t)dt.
\label{xx}
\end{align}	
Let $\varepsilon>0$ be given. There exist  $\delta>0$ such that if $|t-x|<\delta$ then $|\theta(t,x)|<\varepsilon$. We denote
$${\mathcal K}_1=\left\{k:\left|\frac{k}{n}-x\right|<\delta, k=0,1,2,\cdots ,n\right\} \textrm{ and }
{\mathcal K}_2=\left\{k:\left|\frac{k}{n}-x\right|\ge \delta, k=0,1,2,\cdots ,n\right\}.$$
Let $k\in {\mathcal K}_1$. Since $\left| \theta(t,x)\right| < \varepsilon $,  from (\ref{xx}) we obtain
\begin{align}
&\left| \tilde{D}_n^{M,2}\left(\theta(t,x)(t-x)^6;x\right) \right|\nonumber\\
& \leq
(n+1)\left|(n-2)x(x-1)-2x+\frac{3}{2}\right|\sum_{k=0}^{n-2}p_{n-2,k}(x)\int_0^1p_{n,k}(t)|\theta(t,x)|(t-x)^6dt \nonumber\\
&+2(n+1)(n-2)x(1-x)\sum_{k=1}^{n-1}p_{n-2,k-1}(x)\int_0^1p_{n,k}(t)|\theta(t,x)|(t-x)^6dt\nonumber\\
&+(n+1)\left|(n-2)x(x-1)+2x-\frac{1}{2}\right|\sum_{k=2}^np_{n-2,k-2}(x)\int_0^1p_{n,k}(t)|\theta(t,x)|(t-x)^6dt\nonumber\\
&\leq \left[\displaystyle\frac{1}{4}(n+1)(n-2)+\frac{3}{2}(n+1)\right]\varepsilon\sum_{k=0}^{n-2}p_{n-2,k}(x)\int_0^1p_{n,k}(t)(t-x)^6dt \nonumber\\
&+\displaystyle\frac{1}{2}(n+1)(n-2)\varepsilon\sum_{k=1}^{n-1}p_{n-2,k-1}(x)\int_0^1p_{n,k}(t)(t-x)^6dt\nonumber\\
&+\left[\displaystyle\frac{1}{4}(n-2)(n+1)+\frac{3}{2}(n+1)\right]\varepsilon\sum_{k=2}^np_{n-2,k-2}(x)\int_0^1p_{n,k}(t)(t-x)^6dt\nonumber\\
&=\displaystyle\frac{1}{4}(n+1)(n-2)\varepsilon\left\{\sum_{k=0}^{n-2}p_{n-2,k}(x)\int_0^1p_{n,k}(t)(t-x)^6 dt\right.\nonumber\\
&+\left. 2\sum_{k=1}^{n-1}p_{n-2,k-1}(x)\int_0^1p_{n,k}(t)(t-x)^6dt+\sum_{k=2}^np_{n-2,k-2}(x)\int_0^1p_{n,k}(t)(t-x)^6dt\right\}\nonumber\\
&+\displaystyle\frac{3}{2}(n+1)\varepsilon\left\{ \sum_{k=0}^{n-2}p_{n-2,k}(x)\int_0^1p_{n,k}(t)(t-x)^6 dt +\sum_{k=2}^{n}p_{n-2,k-2}(x)\int_0^1p_{n,k}(t)(t-x)^6 dt\right\}\nonumber\\
&=\displaystyle\frac{1}{4}(n+1)(n-2)\varepsilon\left\{\frac{480x^3(1-x)^3n^3}{(n+1)(n+2)(n+3)(n+4)(n+5)(n+6)(n+7)}+{\mathcal O}\left(\frac{1}{n^5}\right)\right\}\nonumber\\
&+\displaystyle\frac{3}{2}(n+1)\varepsilon\left\{\frac{240x^3(1-x)^3n^3}{(n+1)(n+2)(n+3)(n+4)(n+5)(n+6)(n+7)}+{\mathcal O}\left(\frac{1}{n^5}\right)\right\}\nonumber\\
&=\frac{120x^3(1-x)^3n^3(n+1)\varepsilon}{(n+2)(n+3)(n+4)(n+5)(n+6)(n+7)}+{\mathcal O}\left(\frac{1}{n^3}\right)={\cal O}\left(\frac{1}{n^2} \right).\label{ec1}
\end{align}
Let $k\in {\mathcal K}_2$. We denote $M=\displaystyle\sup_{0\leq t\leq 1} |\theta(t,x)|(t-x)^6$. Then $\left| \theta(t,x) \right|(t-x)^6\leq\displaystyle \frac{M}{\delta^6}\left(\frac{k}{n}-x \right)^6$. We get the following upper bound
\begin{align}
&\left| \tilde{D
}_n^{M,2}\left(\theta(t,x)(t-x)^6;x\right) \right|\leq \left[\displaystyle\frac{1}{4}(n-2)+\frac{3}{2}\right]\frac{M}{\delta^6}\sum_{k=0}^{n-2}p_{n-2,k}(x)\left(\displaystyle\frac{k}{n}-x\right)^6\nonumber\\
&+\frac{1}{2}(n-2)\frac{M}{\delta^6}\sum_{k=1}^{n-1}p_{n-2,k-1}(x)\left(\displaystyle\frac{k}{n}-x\right)^6 +\left[\displaystyle\frac{1}{4}(n-2)+\frac{3}{2}\right]\frac{M}{\delta^6}\sum_{k=2}^{n}p_{n-2,k-2}(x)\left(\displaystyle\frac{k}{n}-x\right)^6\nonumber\end{align}\begin{align}
&=\displaystyle\frac{15M(n+1)x^3(1-x)^3}{\delta^6n^3}+{\mathcal O}\left(\frac{1}{n^3}\right)={\cal O}\left(\frac{1}{n^2} \right).\label{ec2}
\end{align}
From (\ref{ec1}) and (\ref{ec2}) the proof of theorem is completed.
\end{proof}

\section{Durrmeyer operators of order III }

Using the approach as in the previous section we construct a third order approximation formula. Let us consider the following modification of Durrmeyer operators:
\begin{equation}\label{K3}
D_n^{M,3}(f;x)=(n+1)\displaystyle\sum_{k=0}^n p_{n,k}^{M,3}(x) \int_{0}^{1}p_{n,k}(t)f(t)dt,
\end{equation}
where
\begin{align*} p_{n,k}^{M,3}(x)&=\tilde{b}(x,n)p_{n-4,k}(x)+\tilde{d}(x,n)p_{n-4,k-1}(x)+\tilde{e}(x,n)p_{n-4,k-2}(x)\\
&+\tilde{d}(1-x,n)p_{n-4,k-3}(x)+\tilde{b}(1-x,n)p_{n-4,k-4}(x) \end{align*}
and
\begin{align*}
& \tilde{b}(x,n)=\tilde{b}_4(n)x^4+\tilde{b}_3(n)x^3+\tilde{b}_2(n)x^2+\tilde{b}_1(n)x+\tilde{b}_0(n),\\
& \tilde{d}(x,n)=\tilde{d}_4(n)x^4+\tilde{d}_3(n)x^3+\tilde{d}_2(n)x^2+\tilde{d}_1(n)x+\tilde{d}_0(n),\\
&\tilde{e}(x,n)=\tilde{e}_0(n)(x(1-x))^2.
\end{align*}
We note that $b_i(n)$, $d_i(n)$, $i=0,1,\dots,4$ and $e_0(n)$ are some unknown sequences which are determined in appropriate way. Let $\tilde{D}_n^{M,3}$ be the operator (\ref{K3}) with the following sequences:
\begin{align*}
& \tilde{b}_0(n)=\displaystyle\frac{10}{3},\,\,\tilde{b}_1(n)=-\frac{53}{3}-\frac{11}{6}n,\,\, \tilde{b}_2(n)=\frac{1}{2}n^2+\frac{29}{6}n+\frac{89}{3},\,\, \tilde{b}_3(n)=-n^2-3n-16,\,\,\tilde{b}_4(n)=\frac{1}{2}n^2,\\
& \tilde{d}_0(n)=\displaystyle-\frac{10}{3},\,\,\tilde{d}_1(n)=\frac{80}{3}+\frac{10}{3}n,\,\,\tilde{d}_2(n)=-2n^2-\frac{28}{3}n-\frac{161}{3},\,\, \tilde{d}_3(n)=4n^2+6n+32,\\
&\tilde{d}_4(n)=-2n^2,\,\, \tilde{e}_0(n)=3n^2.
\end{align*}
\begin{lemma}
The central moments of the operators $\tilde{D}_n^{M,3}$ are given by
\begin{itemize}
	\item [i)] $\tilde{D}_n^{M,3}(t-x;x)=\tilde{D}_n^{M,3}((t-x)^2;x)=\tilde{D}_n^{M,3}((t-x)^3;x)=0$;
	\item [ii)] $\tilde{D}_n^{M,3}((t-x)^4;x)=\displaystyle \frac{20x(1-x)(21x^2-21x+5)n}{\sum_{k=2}^5(n+k)}+{\cal O}\left(\frac{1}{n^4}\right) $;
	\item [iii)] $\tilde{D}_n^{M,3}((t-x)^5;x)=\displaystyle \frac{180(1-2x)x^2(1-x)^2n^2}{\sum_{k=2}^6(n+k)}+{\cal O}\left(\frac{1}{n^4}\right)$;
	\item [v)] $\tilde{D}_n^{M,3}((t-x)^6;x)=\displaystyle\frac{120x^3(1-x)^3n^3}{\sum_{k=2}^{7}(n+k)}+{\cal O}\left(\frac{1}{n^4}\right);  $
	\item [iv)] $\tilde{D}_n^{M,3}((t-x)^7;x)=\tilde{D}_n^{M,3}((t-x)^8;x)=\displaystyle{\cal O} \left(\frac{1}{n^4}\right); $
	\item[v)] $\tilde{D}_n^{M,3}((t-x)^9;x)=\tilde{D}_n^{M,3}((t-x)^{10};x)=\displaystyle{\cal O} \left(\frac{1}{n^5}\right).$
\end{itemize}
\end{lemma}
The asymptotic order of approximation of $\tilde{D}_n^{M,3}$ to $f$ when $n$ goes to infinity is given in the following result:

\begin{theorem}
If $f\in C^{10}[0,1]$ and $x\in [0,1]$, then for sufficiently large $n$, we have
$$ \tilde{D}_n^{M,3}(f;x)-f(x)={\cal O}\left(\displaystyle\frac{1}{n^3}\right). $$
\end{theorem}

\section{Numerical Results}
In this section
 we will analysis the theoretical results presented in the previous sections by numerical examples.
\begin{example} Let $f(x)=\displaystyle\sin(2\pi x)+2\sin\left(\frac{1}{2}\pi x\right)$, $n=10$, $a_0(n)=\displaystyle\frac{n-1}{2n}$ and $a_1(n)=\displaystyle\frac{1}{n}$.
	The convergence of the new modifications of the Durrmeyer operators is illustrated in Figure \ref{fig:1}.  Let $E_n(f;x)=\left| f(x)-D_n(f;x)\right|$ and $E_n^{M,i}(f;x)=\left| f(x)-\tilde{D}_n^{M,i}(f;x)\right|$, $i=1,2,3$ be the error function of the Durrmeyer operators, respectively the modified Durrmeyer operators. The error of approximation is illustrated in Figure \ref{fig:2} and note that the approximation by the modified Durrmeyer operators $\tilde{D}_n^{M,i}$, $i=1,2,3$ is better then using classical Durrmeyer operator $D_n$.
\end{example}

\begin{minipage}{\linewidth}
	\centering
	\begin{minipage}{0.4\linewidth}
		\begin{figure}[H]
			\includegraphics[width=\linewidth]{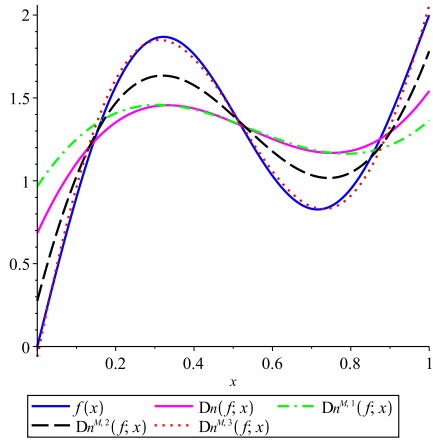}
			\caption{Approximation process}\label{fig:1}
		\end{figure}
	\end{minipage}
	\hspace{0.05\linewidth}
	\begin{minipage}{0.4\linewidth}
		\begin{figure}[H]
			\includegraphics[width=\linewidth]{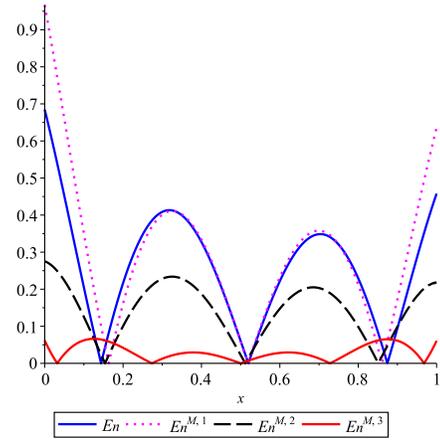}
			\caption{Error of approximation }\label{fig:2}
		\end{figure}
	\end{minipage}
\end{minipage}

\begin{example}
	Let us consider the following function
	$ f(x)=\left| x-\frac{1}{2} \right|\cos(2\pi x). $ It can be observed that $f\in C[0,1]$, but it is not differentiable at the point $x=0.5$. For $n=10$, $a_0(n)=-\displaystyle\frac{n}{2n+1}$ and $a_1(n)=\displaystyle\frac{4n+1}{2n+1}$, the convergence of the modified Durrmeyer operators to  $f(x)$ is illustrated in Figure \ref{fig:3}. Also, the error functions $E_n$ and $E_n^{M,i}$, $i=1,2,3$ are given in Figure \ref{fig:4}.
\end{example}

\begin{minipage}{\linewidth}
	\centering
	\begin{minipage}{0.4\linewidth}
		\begin{figure}[H]
			\includegraphics[width=\linewidth]{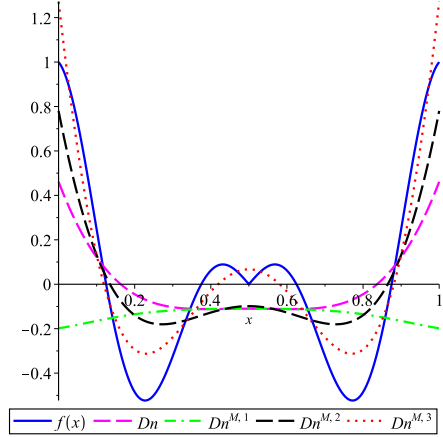}
			\caption{Approximation process}\label{fig:3}
		\end{figure}
	\end{minipage}
	\hspace{0.05\linewidth}
	\begin{minipage}{0.4\linewidth}
		\begin{figure}[H]
			\includegraphics[width=\linewidth]{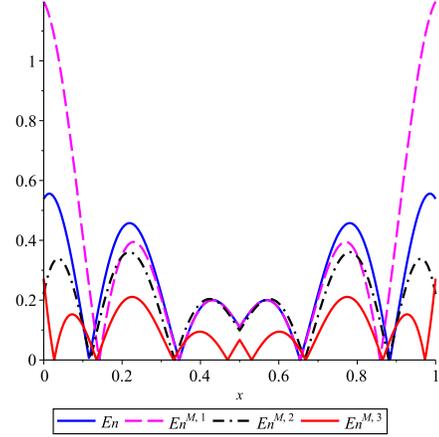}
			\caption{Error of approximation }\label{fig:4}
		\end{figure}
	\end{minipage}
\end{minipage}

\begin{example}
	Let us consider the following function $f(x)=\displaystyle\left(x-\frac{1}{4}\right)\sin(2\pi x)$. The behaviours of the approximations $D_n(f;x)$, $\tilde{D}_n^{M,i}(f;x)$, $i=1,2,3$ and their error functions $E_n(f;x)$, $E_n^{M,i}(f;x)$ for $n=5,10,20$, $a_0(n)=\displaystyle\frac{n-1}{2n}$, $a_1(n)=\displaystyle\frac{1}{n}$ are illustrated in the Figures \ref{fig:5}-\ref{fig:12}.
\end{example}

\begin{minipage}{\linewidth}
	\centering
	\begin{minipage}{0.4\linewidth}
		\begin{figure}[H]
			\includegraphics[width=\linewidth]{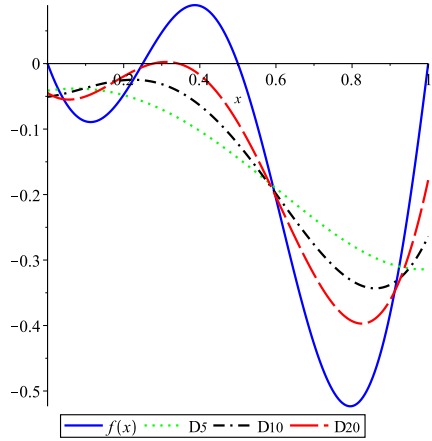}
			\caption{Approximation process of $D_n$}\label{fig:5}
		\end{figure}
	\end{minipage}
	\hspace{0.05\linewidth}
	\begin{minipage}{0.4\linewidth}
		\begin{figure}[H]
			\includegraphics[width=\linewidth]{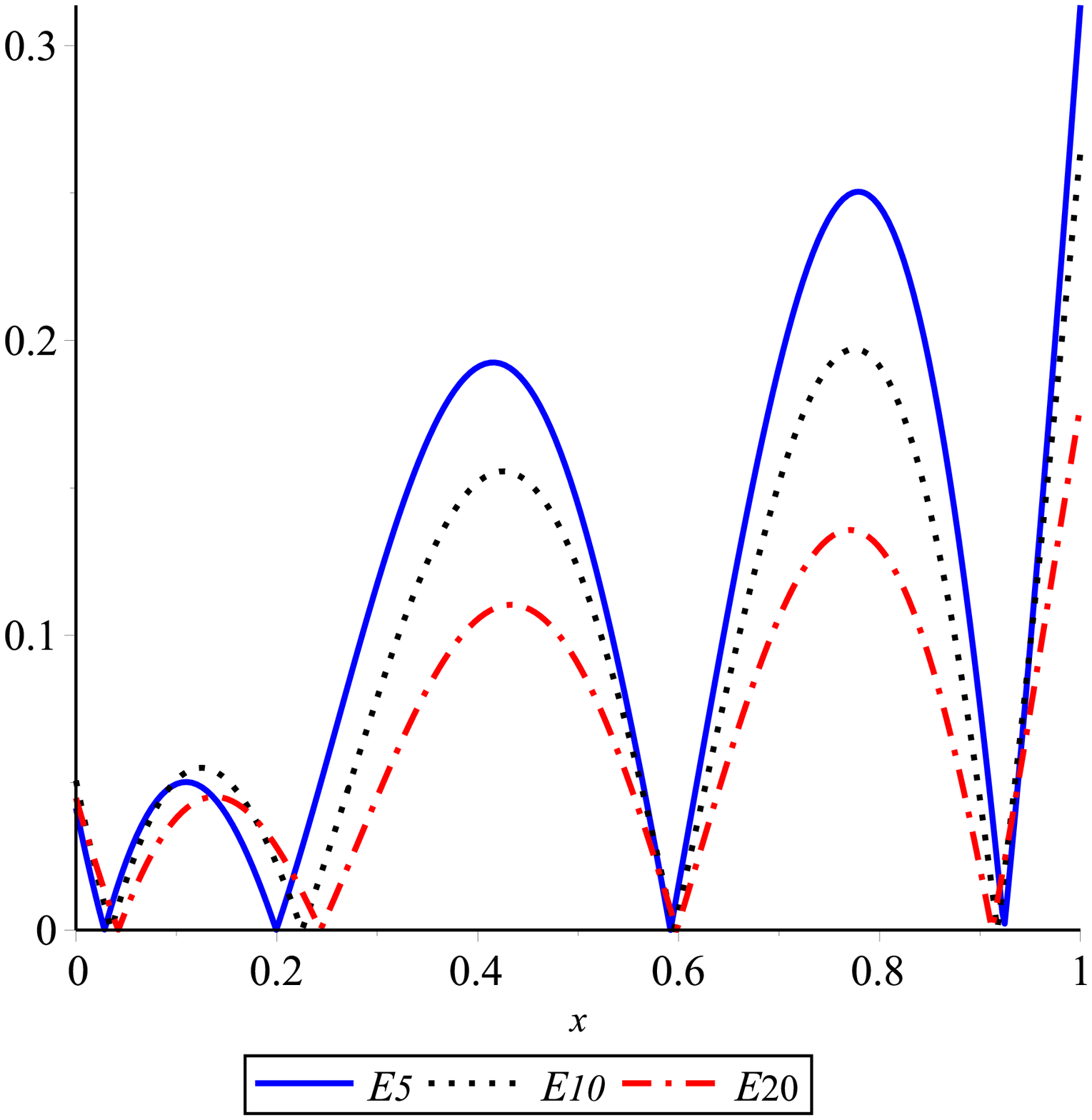}
			\caption{Error of approximation $E_n$ }\label{fig:6}
		\end{figure}
	\end{minipage}
\end{minipage}

\begin{minipage}{\linewidth}
	\centering
	\begin{minipage}{0.4\linewidth}
		\begin{figure}[H]
			\includegraphics[width=\linewidth]{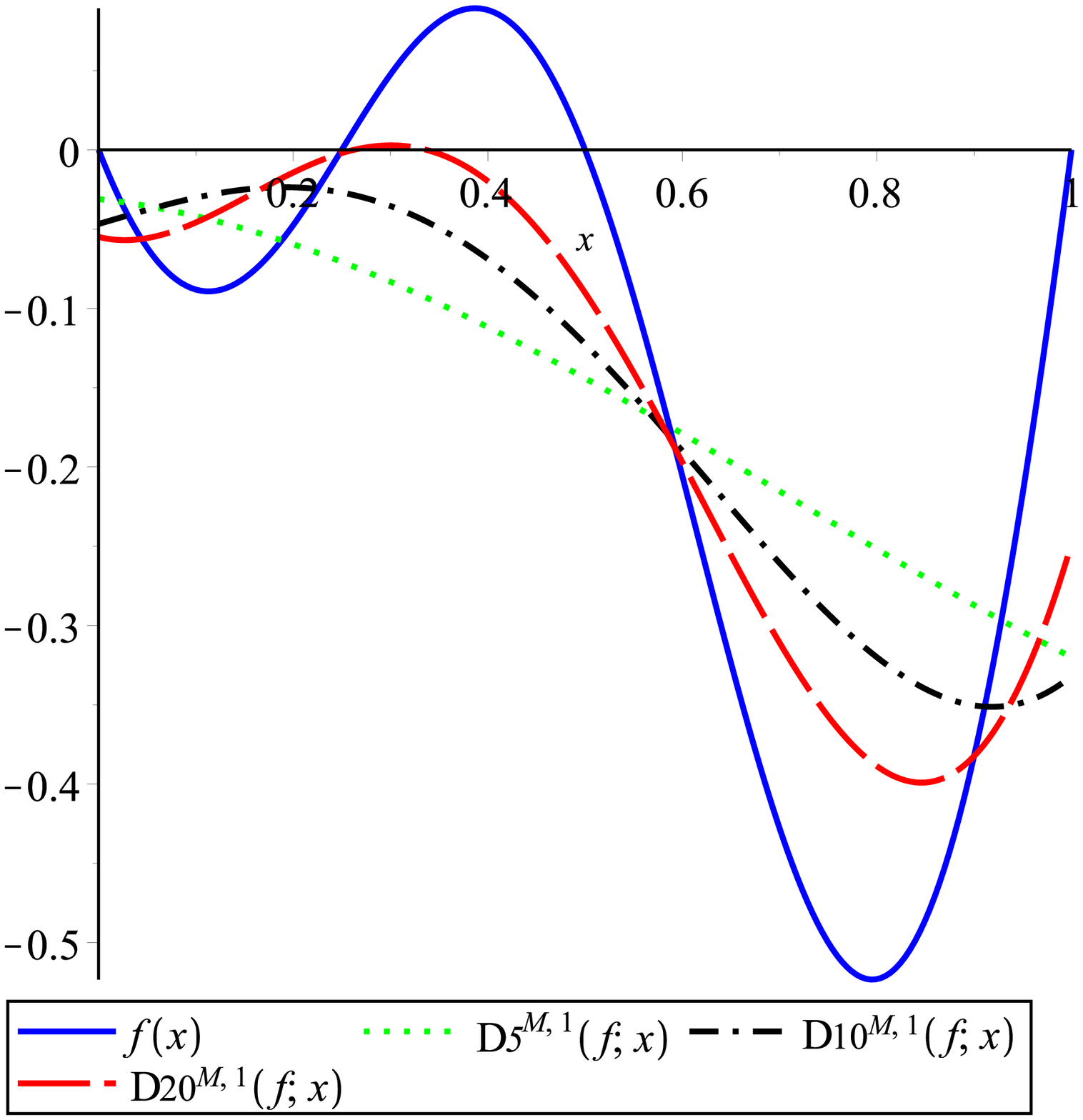}
			\caption{Approximation process of $\tilde{D}_n^{M,1}$}\label{fig:7}
		\end{figure}
	\end{minipage}
	\hspace{0.05\linewidth}
	\begin{minipage}{0.4\linewidth}
		\begin{figure}[H]
			\includegraphics[width=\linewidth]{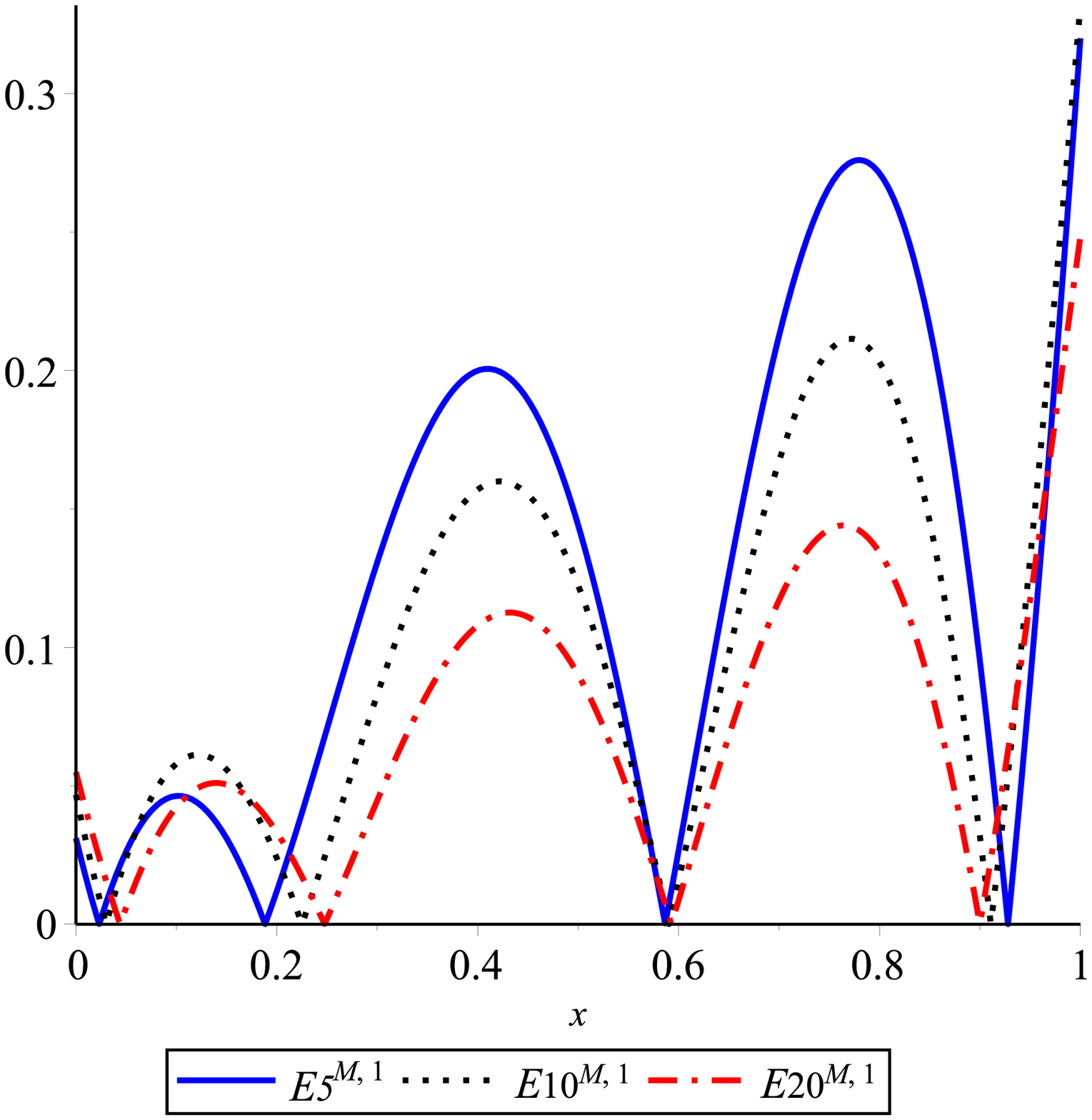}
			\caption{Error of approximation $E_n^{M,1}$ }\label{fig:8}
		\end{figure}
	\end{minipage}
\end{minipage}

\begin{minipage}{\linewidth}
	\centering
	\begin{minipage}{0.4\linewidth}
		\begin{figure}[H]
			\includegraphics[width=\linewidth]{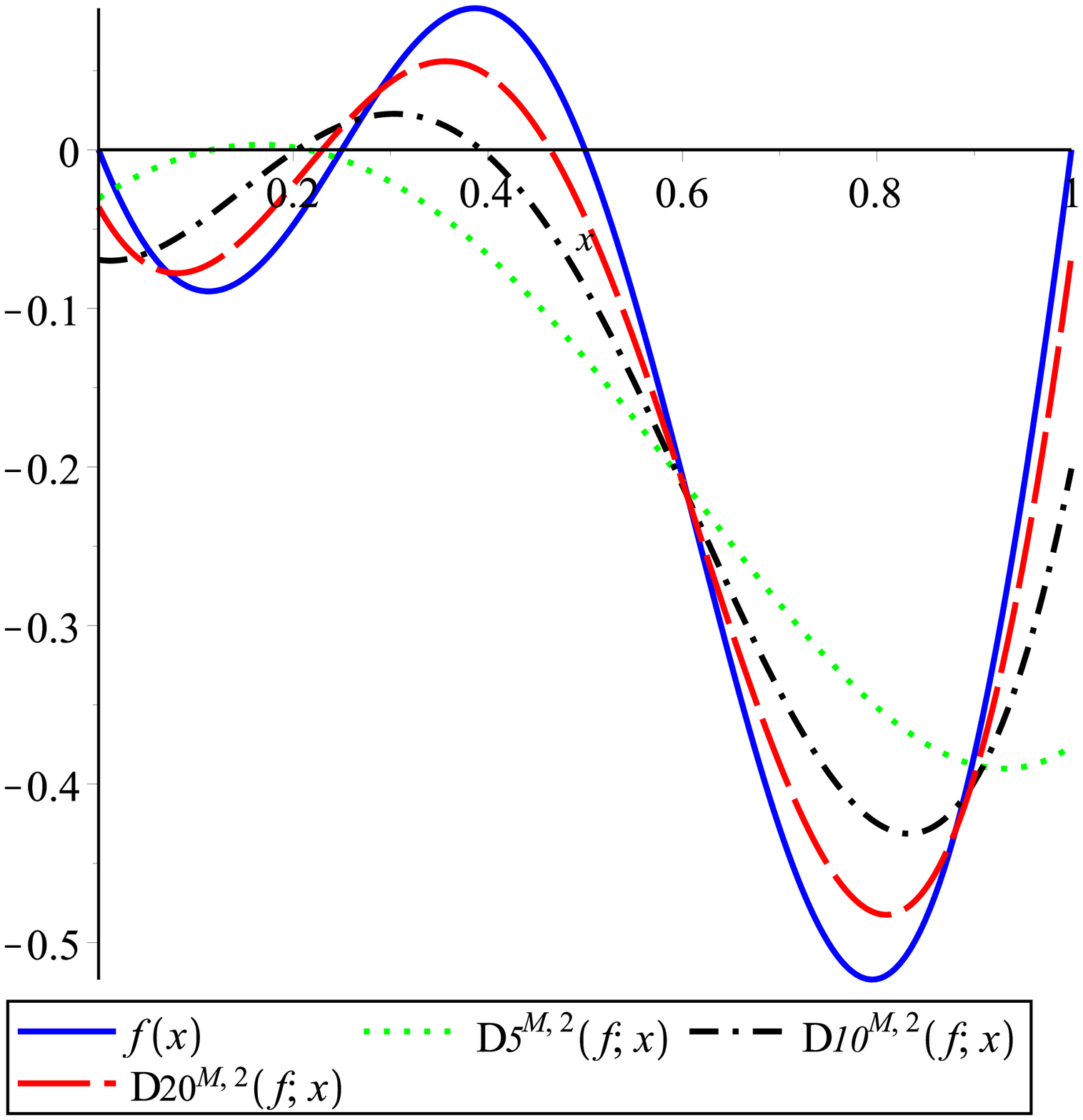}
			\caption{Approximation process of $\tilde{D}_n^{M,2}$}\label{fig:9}
		\end{figure}
	\end{minipage}
	\hspace{0.05\linewidth}
	\begin{minipage}{0.4\linewidth}
		\begin{figure}[H]
			\includegraphics[width=\linewidth]{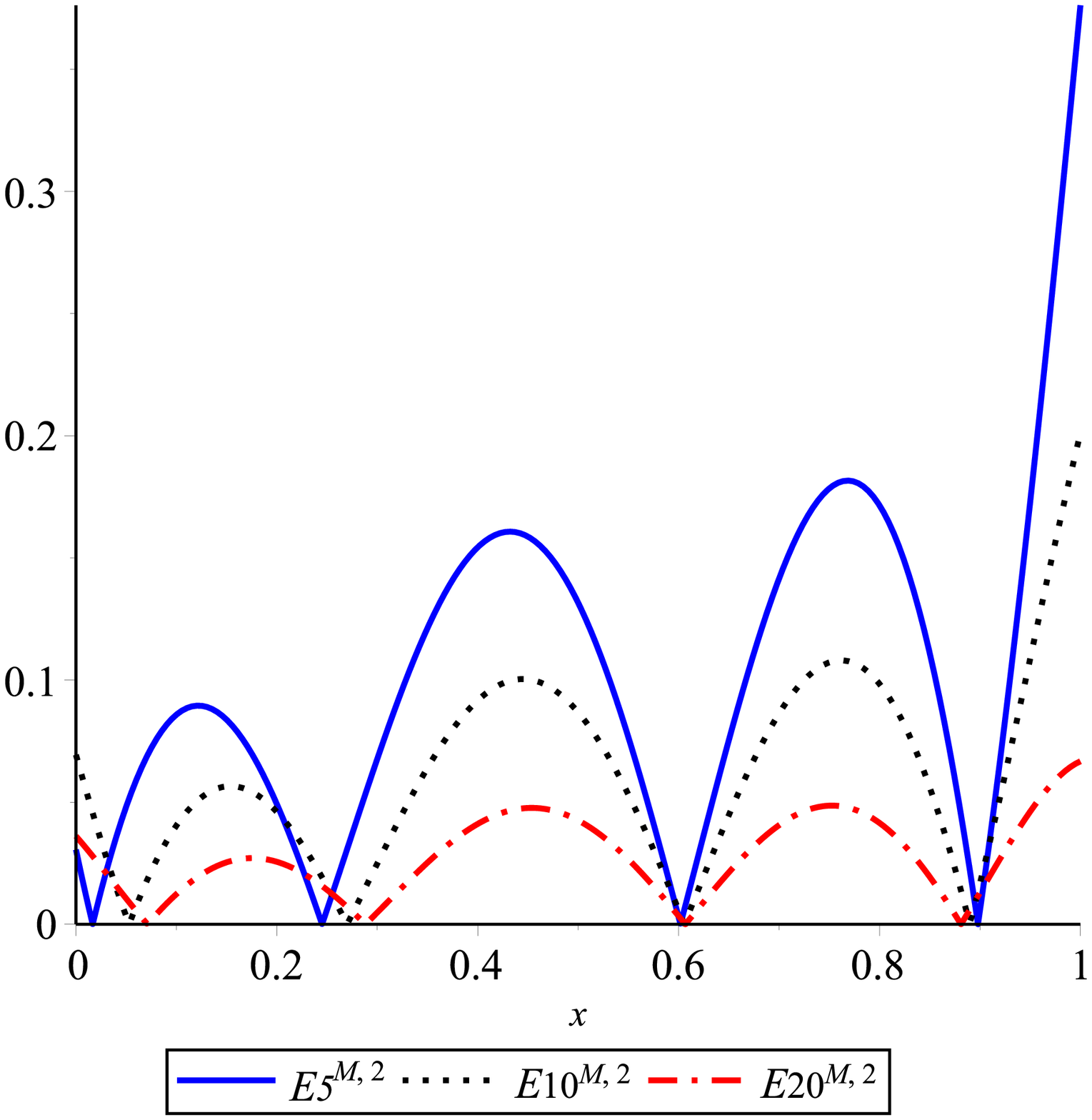}
			\caption{Error of approximation $E_n^{M,2}$ }\label{fig:10}
		\end{figure}
	\end{minipage}
\end{minipage}

\begin{minipage}{\linewidth}
	\centering
	\begin{minipage}{0.4\linewidth}
		\begin{figure}[H]
			\includegraphics[width=\linewidth]{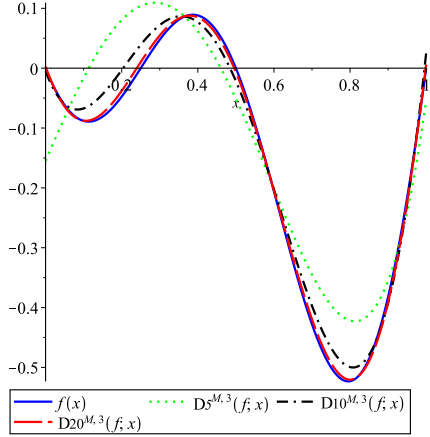}
			\caption{Approximation process $\tilde{D}_n^{M,3}$}\label{fig:11}
		\end{figure}
	\end{minipage}
	\hspace{0.05\linewidth}
	\begin{minipage}{0.40\linewidth}
		\begin{figure}[H]
			\includegraphics[width=\linewidth]{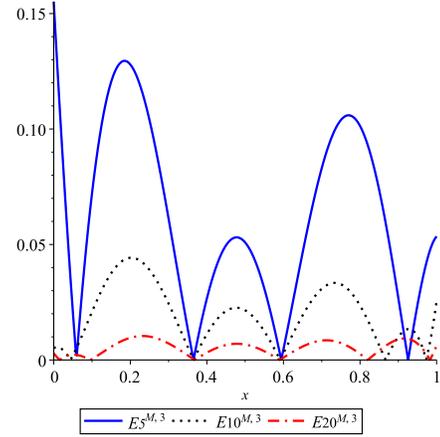}
			\caption{Error of approximation $E_n^{M,3}$ }\label{fig:12}
		\end{figure}
	\end{minipage}
\end{minipage}

$$    $$

\noindent{\bf Acknowledgements.} Project financed from Lucian Blaga University of Sibiu research grants LBUS-IRG-2018-04.

$   $


\noindent{\bf References}

\begin{thebibliography}{99}
	\bibitem{1} H. Khosravian-Arab, M. Dehghan, M. R. Eslahchi, {\it A new approach to improve the order of approximation of the Bernstein
	operators: Theory and applications}, Numerical Algorithms 77(1) (2018), 111-150.
	
\bibitem{Bernstein}  S.N. Bernstein, {\it D\'{e}monstration du th\'{e}or\`{e}me de Weierstrass fond\'{e}e
	sur le calcul des probabilit\'{e}s}, Communications de la
Soci\'{e}t\'{e} Mathematique de Kharkov, 13, 1913, 1--2.
\bibitem{bust} J. Bustamante, {\it  Bernstein
Operators and Their Properties}, Birkh\"{a}user, Basel, 2017.


\bibitem{8r3}  M.M. Derriennic, {\it Sur lâ approximation des fonctions int\'egrables par des polyn\^omes de Bernstein
modifi\'es}, J. Approx. Theory 31 (1981), 325--343.

\bibitem{10r3}  Z. Ditzian, {\it  Multidimensional Jacobi-type Bernstein-Durrmeyer operators}, Acta. Sci. 60
(1995), 225--243.

\bibitem{1a}  J.L. Durrmeyer, {\it Une formule d'inversion de la transforme de Laplace: Applications a la theorie des moments},
These de 3e cycle, Paris, (1967).

\bibitem{11r3} H. Esser, {\it On pointwise convergence estimates for positive linear operators on $C[a,b]$}, Indag.
Math. 38 (1976), 189-194.

\bibitem{Gavrea} I. Gavrea, M. Ivan, {\it An answer to a conjecture on Bernstein operators}, J. Math. Anal. Appl. 390 (1) (2012), 86--92.

\bibitem{2} H. Gonska, I. Rasa, {\it Asymptotic behavior of differentiated Bernstein polynomials}, Mat. Vesnik 61 (1) (2009), 53--60.

\bibitem{3} H. Gonska, {\it On the degree of approximation in Voronovskaja's theorem},  Studia Univ. Babe\c{s} Bolyai Math. 52 (3) (2007), 103--115.

\bibitem{21r2} H. Gonska, {\it Two problems on best constants in direct estimates}, (Problem Section of Proc.
Edmonton Conf. Approximation Theory, ed. by Z. Ditzian et al.), 194 Providence, RI. Amer.
Math. Soc. (1983).

\bibitem{27r2}   H. Gonska, D. Zhou, On an extremal problem concerning Bernstein operators, Serdica Math.
J. (1995), 137--150.


\bibitem{G} H. Gonska, {\it Quantitative Aussagen zur Approximation durch positive lineare Operatoren}, Ph. D. Thesis, Duisburg, Universit at Duisburg, 1979.

\bibitem{G1}  V. Gupta and G. Tachev, {\it Approximation with Positive Linear Operators and Linear Combinations}, Springer, Cham, 2017.

\bibitem{G2}  V. Gupta and R. P. Agarwal, {\it Convergence Estimates in Approximation Theory}, Springer, Cham, 2014.
\bibitem{22r3}  K. G. Ivanov, {\it On Bernstein polynomials}, C.R. Acad. Bulg. 35 (1982), 893--896.



\bibitem{30r2}  D. P. Kacs\'o, Discrete Jackson type operators via a Boolean sum approach, J. Comp. Anal.
Appl., 3 (2001), 399-413.

\bibitem{32r3}  G. G. Lorentz, {\it Bernstein polynomials}, Toronto: Univ. Press 1953.

\bibitem{2a}  A. Lupa\c s,{\it Die Folge der Betaoperatoren, Dissertation}, Universit\"ait Stuttgart, (1972).

\bibitem{50r2}   R. P\u alt\u anea, {\it Approximation by linear positive operators: Estimates with second order moduli},
Bra\c sov: Ed. Univ. Transilvania 2003.

\bibitem{P43} T. Popoviciu, {\it Sur l'approximation des fonctions covexes d'ordre sup\'{e}rieur}, Mathematica (Cluj)
10 (1934), 49–54.

\bibitem{P44} T. Popoviciu, {\it Sur l'approximation des fonctions continues d'une variable r\'{e}elle par des polynomes},
Ann. Sci. Univ. Ia¸si, Sect. I, Math 28 (1942), 208.

\bibitem{58r3}  P. C. Sikkema, {\it Der Wert einiger Konstanten in der Theorie der Approximation mit Bernsteinâ
	Polynomen}, Num. Math. 3 (1961), 107-116.

\bibitem{5} G. Tachev, {\it The complete asymptotic expansion for Bernstein operators}, J. Math. Anal. Appl. 385 (2) (2012), 1179--
1183.

\bibitem{T9}  E. Voronovskaja, {\it  Determination de la forme asymptotique d'approximation des fonctions par les polynomes de M. Bernstein}, Dokl. Akad. Nauk SSSR 4 (1932), 86--92.

\end{thebibliography}
\end{document}